\documentclass[oneside,11pt]{amsart}

\makeatletter
\@namedef{subjclassname@2020}{
  \textup{2020} Mathematics Subject Classification}
\makeatother
\usepackage[a4paper,width=170mm,top=27mm,bottom=27mm]{geometry}
\usepackage{cases}
\usepackage{amsmath,amssymb}
\usepackage[hidelinks]{hyperref}
\usepackage{esint}

\usepackage[french,english]{babel}

\newtheorem{theorem}{Theorem}[section]
\newtheorem{lemma}[theorem]{Lemma}

\newtheorem{proposition}[theorem]{Proposition}

\theoremstyle{definition}
\newtheorem{remx}[theorem]{Remark}
\newenvironment{remark}
  {\pushQED{\qed}\remx}
  {\popQED\endremx}

\newcommand{\pt}{\partial}
\newcommand{\C}{\mathbb{C}}
\newcommand{\R}{\mathbb{R}}
\newcommand{\Z}{\mathbb{Z}}
\newcommand{\N}{\mathbb{N}}
\newcommand{\T}{\mathbb{T}}

\newcommand{\vare}{\varepsilon}
\newcommand{\ld}{\lambda}

\numberwithin{equation}{section}

\begin{document}
\address{Yongming Luo
\newline \indent
Faculty of Computational Mathematics and Cybernetics
\newline \indent Shenzhen MSU-BIT University, China}
\email{luo.yongming@smbu.edu.cn}

\address{Jiayu Zheng
\newline \indent
Faculty of Computational Mathematics and Cybernetics
\newline \indent Shenzhen MSU-BIT University, China}
\email{jyzheng@smbu.edu.cn}

\title[Unconditional uniqueness for the cubic NLS]
{Unconditional uniqueness for the cubic nonlinear Schr\"odinger equation in $H^{1/2}(\T^3)$}
\author{Yongming Luo and Jiayu Zheng}

\begin{abstract}
In this paper, we establish unconditional uniqueness for the cubic nonlinear Schrödinger equation (NLS) in the scaling-critical class \(H^{1/2}(\mathbb T^3)\). Unlike approaches based on the GP-hierarchy or normal form reductions, which typically involve higher-order expansions, our proof works directly with the original cubic NLS, where we also make use of several new \textit{a priori} Euclidean-like Strichartz estimates. More precisely, we show that every distributional solution in \(C_tH_x^{1/2}\) belongs \textit{a priori} to the class \(L^2_{t,\text{loc}}L_x^6\). The key ingredients are a suitable high–low modulation decomposition and the compactness of solution trajectories on bounded time intervals. We also exploit the preservation of the canonical symplectic form under the truncated NLS flow, in order to obtain necessary stability estimates in the rough Sobolev space \(H^{-1/2}(\mathbb T^3)\).
\end{abstract}


\maketitle

\section{Introduction}
Consider the cubic nonlinear Schr\"odinger equation (NLS) 
\begin{equation}\label{cnls}
(i\pt_t +\Delta) u=\mu |u|^2 u,\qquad \mu=\pm 1
\end{equation}
on $\T^3$. We aim to prove the following unconditional uniqueness result for the cubic NLS in the scaling critical space $H^{1/2}(\T^3)$.

\begin{theorem}\label{main thm}
Let $I\subset \R$ be an interval and suppose that $u,v\in C(I; {H}^{1/2}(\T^3))$ are distributional solutions of the cubic NLS \eqref{cnls} satisfying $u(t_0)=v(t_0)$ for some $t_0\in I$. Then $u(t)=v(t)$ for any $t\in I$.
\end{theorem}

\subsection{Historical background}
The question of unconditional uniqueness for the solutions of the NLS arises naturally from
the classical local well-posedness theory. In general, a contraction argument
constructs solutions in a suitable auxiliary space, which
combines time continuity in a Sobolev space with additional
space-time integrability. Although such an argument gives
uniqueness within the auxiliary space, it does not necessarily
rule out another distributional solution with the same initial
data, that is only known to belong to the natural Sobolev class.
Unconditional uniqueness aims to remove these auxiliary
assumptions and establish uniqueness in the Sobolev class
itself.

An important motivation for studying this property is to
identify solutions obtained by different constructions.
Indeed, solutions arising from contraction schemes,
compactness arguments or other methods may not be known
to satisfy the additional space-time bounds required by
the classical local theory. Unconditional uniqueness allows
such solutions to be identified without separately verifying
these bounds. This issue is also relevant to variational
and control problems, where the natural function space
for constructing solutions need not coincide with the
auxiliary space used in the evolution theory; see e.g. \cite{LugiatoLefever1987} and also
the discussion in \cite{uncond2}.

A systematic study of unconditional uniqueness for NLS
might at least date back to Kato \cite{Kato1995}. The main idea of
Kato's approach is to consider the difference of two
solutions with the same initial data and establish a
short-time bootstrap estimate. Since the difference has
zero initial data, its Duhamel representation contains
no free evolution term. One then seeks to estimate the
nonlinear difference by a small multiple of the norm
of the difference itself, so that uniqueness follows
by absorption. A crucial requirement is that the norm
used in this argument must already be finite under
the assumed Sobolev regularity, rather than imposed
as an additional condition on the solutions.

At low regularity, however, Kato's approach encounters a
fundamental difficulty arising from the mismatch between
Sobolev embedding and Strichartz estimates. To overcome
this difficulty, Furioli and Terraneo
\cite{FurioliTerraneo2003} developed Kato's bootstrap
argument by estimating the difference of two solutions
in Besov spaces of negative regularity. The main idea
is to trade spatial differentiability for integrability,
thereby allowing the nonlinear estimates to be matched
with the available Strichartz estimates. This approach
significantly extends the range of regularities for
which unconditional uniqueness can be established.
For further results in this direction, we refer to
\cite{HanFang2013,Rogers2007,LuXu2016}. 

For the cubic NLS \eqref{cnls} on $\R^3$, it is also worth pointing out that the result of Furioli and Terraneo
\cite{FurioliTerraneo2003} already yields unconditional uniqueness for every $1/2<s<1$, while the scaling-critical
case $s=1/2$ had remained open. This open endpoint was recently resolved in our work \cite{LuoZheng2026}, by appealing to several novel tools such as  
the endpoint critical Besov space $\dot B^{-1}_{\infty,\infty}(\R^3)$ and suitable Born-type Duhamel expansion.

\subsection{New challenges and methods in the periodic setting}
In previous works, the unconditional uniqueness problem was mainly studied on the Euclidean space $\R^d$. The problem becomes substantially more delicate and difficult when the underlying domain is replaced by a torus $\T^d$ or more generally a compact manifold. Unlike the Euclidean setting,
where the free Schr\"odinger propagator enjoys strong
dispersive decay, the NLS-flow on compact manifolds does usually not admit a
global-in-time dispersive estimate.
Consequently, Strichartz estimates on compact manifolds are typically formulated in term of
frequency-localized functions or eigenfunctions, see e.g. \cite{Bourgain1,l2decoupling,Burq1,Burq2,Burq3} and the references therein. These restrictions thus make it rather difficult to directly adapt the Euclidean approaches to the compact settings.

To overcome these difficulties, several approaches
have been developed by exploiting additional structures
of the NLS beyond the usual Strichartz estimates.
One such approach is based on the Gross--Pitaevskii
(GP) hierarchy, which arises naturally in the study
of quantum many-body systems. The key observation is
that a solution of the NLS generates a sequence
of factorized density matrices satisfying the
corresponding GP-hierarchy, and therefore uniqueness for this
infinite system can then be used to establish
uniqueness for the original NLS. The hierarchy
formulation allows one to iteratively expand the
interaction terms through Duhamel's formula and
organize the resulting multilinear expressions
using suitable combinatorial and analytic estimates. In this direction, Herr and Sohinger \cite{HerrSohinger2019} used for the first time the GP-hierarchy framework 
to establish unconditional uniqueness results for the cubic NLS on rectangular
tori. As a consequence, it is proved in \cite{HerrSohinger2019} that a solution of \eqref{cnls} is unconditionally unique in $H^{s}(\T^3)$ for any $s>4/5$.

A significant further advance was made by Chen and Holmer
\cite{uncond1}, who established unconditional uniqueness
for the energy-critical quintic NLS in $H^1(\T^3)$,
without any smallness assumption on the solutions.
The key ingredient of their argument is a new
hierarchical uniform frequency localization (HUFL)
property for the GP-hierarchy, which allows
one to exploit the uniform smallness of the high-frequency
part of a solution in the iterated Duhamel estimates,
thereby providing the smallness needed to close the
uniqueness argument at critical regularity.
In a subsequent work \cite{uncond2}, Chen and Holmer
further resolved the unconditional uniqueness problem
for the energy-critical cubic NLS in $H^1(\T^4)$.
An additional difficulty arises in this setting,
due to the fact that the Sobolev multilinear estimates are no longer sufficient.
Suitable $U^p$-$V^p$ multilinear estimates, together with an
extended Klainerman--Machedon boardgame argument, are exploited to overcome the new difficulties. We also refer to \cite{uncond3} for unconditional uniqueness results for energy-supercritical NLS.

Another important approach is the infinite normal
form reduction, which exploits the oscillatory
structure of the nonlinear interactions in
Fourier space. Instead of estimating the
nonlinearity directly through Strichartz estimates, which are generally incompatible with NLS of low critical scaling,
the crucial point of the normal form method lies in the fact, that higher order expansion of the nonlinearity can be obtained via repeatedly integration by parts in the nonresonant interactions. The resulting
inverse modulation factors hence provide additional
control, while the resonant terms are treated
separately. In this direction, we refer e.g. to \cite{GuoKwonOh2013,Kishimoto2021} and the references therein.

It is worth pointing out that Theorem \ref{main thm} is not covered by the previous results. The main difficulty is that for models of low critical regularity, such as the critical scaling $s=1/2$ considered in the present paper, the
periodic Strichartz estimates are still too restrictive
to provide sufficient space-time estimates. Overcoming this difficulty requires a new way
to recover additional space-time integrability
directly from the equation, which we elaborate in more details in the next subsection.

\subsection{Strategy for proving the main result}

We now explain the main ideas behind the proof of
Theorem \ref{main thm}. Unlike the GP-hierarchy and
normal form approaches discussed above, our argument
does not rely on higher-order expansions of the
nonlinear interactions. Instead, we work directly
with the original cubic NLS and exploit the structure
of its Hamiltonian flow. The main idea is to show that
every distributional solution in $C(I;H^{1/2}(\T^3))$
possesses sufficient additional space-time integrability
to be identified with the solution constructed by
the critical local well-posedness theory. The uniqueness
then follows from the uniqueness of the latter solution.

The first and most important step is to establish the
\textit{a priori} estimate
\[
u\in L_{t,\mathrm{loc}}^2L_x^6
\]
for every distributional solution
$u\in C_t H_x^{1/2}$.
This estimate is not a direct consequence of the
usual periodic Strichartz theory. Indeed, the key point is to exploit a delicate decomposition according
to the modulation relative to the characteristic
surface $\tau+|n|^2=0$. More precisely, at spatial
frequency $|n|\sim N$, we separate the low-modulation
region $|\tau+|n|^2|\lesssim N$ from its complement. The two regions will then be treated by different approaches.

For the low-modulation part, we shall make use of the
frequency-localized Strichartz estimate introduced by Burq, G\'{e}rard and Tzvetkov \cite{Burq1} on time
intervals of length $O(N^{-1})$. In particular, the short-time Strichartz estimates of \cite{Burq1}
allow us to work with admissible pairs having smaller time integrability
exponents, such as $(2,6)$, which lie beyond the range covered by the
usual Bourgain-type Strichartz estimates on $\T^3$ \cite{Bourgain1}.
Together with Littlewood--Paley square-function
estimates, this gives the desired $L_t^2L_x^6$
bound at the cost of only half a spatial derivative.

For the high-modulation part, the equation can be
inverted away from the characteristic surface,
giving rise to a Fourier multiplier containing
the factor $(\tau+|n|^2)^{-1}$.
We establish an $L_t^2L_x^{6/5}$ to $L_t^2L_x^6$
estimate for this operator by combining Euclidean
endpoint Strichartz estimates with a suitable
periodization argument. 

We note here that the compactness of the solution trajectories also play a crucial role. Indeed, since
$u\in C(I;H^{1/2}(\T^3))$, the set
$\{u(t):t\in J\}$ is compact in $H^{1/2}(\T^3)$
for every compact interval $J\subset I$.
Consequently, given any $\varepsilon>0$, one
can choose a frequency cutoff $K$ such that
\[
\sup_{t\in J}
\|(I-P_{\le K})u(t)\|_{H_x^{1/2}}
<\varepsilon.
\]
We then write
\[
|u|^2
=
|P_{\le K}u|^2
+
\bigl(|u|^2-|P_{\le K}u|^2\bigr).
\]
The low-frequency coefficient is bounded in
$L_{t,x}^\infty$, while the second term is
arbitrarily small in $L_t^\infty L_x^{3/2}$.
Combining this decomposition with the
high-modulation estimate, we obtain a fixed-point
equation in which the contribution of the
high-frequency coefficient defines a contraction
on $L_t^2L_x^6$.

Nevertheless, it still remains a subtle point to be handled in this argument.
The contraction only produces a
solution belonging to the space $L_t^2L_x^6$, whereas the
original distributional solution is initially
known merely to lie in $C_tH_x^{1/2}$.
We therefore can not identify these two solutions
by the strong estimate alone, since this would
already require the additional integrability
we are trying to prove. It is at this point where the Lorentz-space
estimates enter our argument. Indeed, 
we establish a
second contraction in the weaker space characterized
by
\[
(1-\Delta)^{-1/5}u
\in L_t^\infty L_x^{5,\infty},
\]
which is already controlled by the assumed
$L_t^\infty H_x^{1/2}$ regularity. The resulting contraction in the Lorentz class then identifies the original
distributional solution with the fixed point
constructed in $L_t^2L_x^6$, thereby establishing
the desired \textit{a priori} space-time integrability. At the end, we shall also make use of the preservation of the canonical symplectic form, given as 
\[
\omega(g_1,g_2)
=\operatorname{Im}\int_{\T^3}
g_1\overline{g_2}\,dx,\qquad\omega\bigl(
D_f\Phi_{L,t}(f)g_1,
D_f\Phi_{L,t}(f)g_2
\bigr)
=
\omega(g_1,g_2),
\]
to deduce the useful uniform bound
\[
\|D_f\Phi_{L,t}(f)\|_{H^{-1/2}\to H^{-1/2}}
\lesssim1
\]
which closes our desired argument in the rough Sobolev space $H_x^{-1/2}$.

\subsubsection*{Outline of the paper} 
The rest of the paper is organized as follows. In Section \ref{sec2} we show that the proof of Theorem \ref{main thm} can be reduced to proving two key Propositions \ref{prop 2.1} and \ref{prop 2.2}. The proofs of the key propositions will then be given in Section \ref{sec3} and Section \ref{sec4} respectively. 

\subsection{Notation}
We use the notation $A\lesssim B$ whenever there exists some positive constant $C$ such that $A\leq CB$. Similarly we define $A\gtrsim B$ and use $A\sim B$ when $A\lesssim B\lesssim A$. We use notation such as $P_N$ or $P_{\leq L}$ to denote the standard Littlewood-Paley projection for dyadic numbers $N$ and $L$. We use $L^p$, $L^{p,q}$ and $H^s$ to denote the usual Lebesgue, Lorentz and Sobolev spaces, respectively. 

\subsection{AI disclosure}
The authors used GPT-6 Astra (OpenAI) for assistance in mathematical reasoning, the initial development of the proof, and manuscript preparation. However, a first AI-assisted version of drafts and proofs required extensive revision by the authors, both mathematically and in its presentation.

One major difficulty with the AI-generated steps was its exposition. It frequently provided lengthy and repetitive explanations of elementary arguments, while certain nontrivial steps were treated too briefly or without sufficient justification. The authors therefore substantially reorganized and rewrote the proof, removing unnecessary discussions of routine arguments and supplying detailed justifications for steps requiring genuine mathematical analysis.

A more substantive mathematical revision concerned the choice of function space for the contraction argument. The initial AI-assisted proof was formulated primarily in the \(L_{t,x}^4\) norm. This approach would have required additional analytical tools, including local smoothing estimates of Staffilani and Tataru \cite{StaffTataru}, the endpoint inhomogeneous Strichartz estimates of Koh and Seo \cite{KohSeo}, and Khintchine's inequality. The authors subsequently recognized that the contraction argument could instead be formulated in the \(L_t^2L_x^6\) norm. This observation eliminated the need for these additional tools and led to a considerably simpler and more direct proof.

Throughout the preparation of the manuscript, all AI-generated mathematical arguments were critically examined and independently verified by the authors. The authors made the necessary mathematical corrections, developed the aforementioned simplification, and substantially revised the exposition. They take full responsibility for the correctness, originality, and final content of the manuscript.

\section{Proof of Theorem \ref{main thm}}\label{sec2}
In this section, we give the proof of Theorem \ref{main thm}. More precisely, we show that it suffices to establish two key propositions, which we state below.

\begin{proposition}\label{prop 2.1}
Let $I\subset \R$ be a nonempty interval and suppose that $u\in C(I; H^{1/2}(\T^3))$ is a distributional solution of \eqref{cnls}. Then for any compact $J\subset I$ it holds
$
u\in L_t^3 L_x^{9/2}(J).
$
\end{proposition}

\begin{proposition}\label{prop 2.2}
Let $\mathcal K \subset H^{1/2}(\mathbb T^3)$ be compact.
Then there exist an open neighborhood $\mathcal U$ of
$\mathcal K$ in $H^{1/2}(\mathbb T^3)$, a time $h>0$,
and a constant $C<\infty$ such that, for every
$L\in\mathbb N\cup\{\infty\}$, the Cauchy problem
\begin{equation}
\begin{cases}
(i\partial_t+\Delta)z
=\mu P_{\le L}\bigl(
|P_{\le L}z|^2P_{\le L}z
\bigr),\\
z(0)=f\in\mathcal U,
\end{cases}
\label{2.2}
\end{equation}
admits a solution
$\Phi_{L,\cdot}(f)\in
C([-h,h];H^{1/2}(\mathbb T^3))$
constructed by the critical local well-posedness theory.
Moreover, the following properties hold:

\begin{itemize}
\item The flow map
$f\mapsto\Phi_{L,t}(f)$ defines a real $C^1$ map on $H^{\frac12}(\T^3)$ and the functional derivative $D_f\Phi_{L,t}(f)$ satisfies
\begin{equation}
\begin{aligned}
\sup_{|t|\leq h,\,f\in \mathcal{U},\,L\in\N}\left(\|D_f\Phi_{L,t}(f)\|_{H_x^{1/2}\to H_x^{1/2}}+
\|[D_f\Phi_{L,t}(f)]^{-1}\|_{H_x^{1/2}\to H_x^{1/2}}\right)
\lesssim 1.
\end{aligned}
\label{2.3}
\end{equation}

\item
The truncated flow $\Phi_{L,t}$ converges to the original solution flow $\Phi_t$ associated with \eqref{cnls} in the sense that
\begin{equation}
\sup_{\substack{|t|\le h,\,f\in\mathcal K}}
\|\Phi_{L,t}(f)-\Phi_t(f)\|_{H_x^{1/2}}
\longrightarrow 0
\qquad\text{as }L\to\infty.
\label{2.5}
\end{equation}

\item 
The functional derivative
$D_f\Phi_{L,t}(f)$ extends to a bounded
real-linear operator on $H^{-1/2}(\mathbb T^3)$ and satisfies 
\begin{equation}
\sup_{|t|\leq h,\,f\in \mathcal{U},\,L\in\N}\|D_f\Phi_{L,t}(f)\|_{H_x^{-1/2}\to H_x^{-1/2}}
\lesssim 1.
\label{2.4}
\end{equation}
\end{itemize}
\end{proposition}

\begin{remark}
It is worth pointing out that \eqref{2.3} and \eqref{2.4} hold for any $f\in\mathcal U$, while in \eqref{2.5} we shall merely assume $f\in \mathcal K$. The reason is that the former estimates are stable under small perturbation (hence extend to hold on a neighborhood), whereas the convergence in \eqref{2.5} is generally only uniform on compact sets.
\end{remark}

With Proposition \ref{prop 2.1} and \ref{prop 2.2} in place, we give the proof of Theorem \ref{main thm}.

\begin{proof}[Proof of Theorem \ref{main thm}]
Let $u\in C(I;H_x^{1/2})$ be a distributional solution of \eqref{cnls} and $J=[a,b]\subset I$ be a compact interval. We first show that
$$u(t)=\Phi_{t-s}(u(s))$$ 
whenever $s,t\in J$ are sufficiently close. Define the set
\[
\mathcal K
=\overline{\bigl\{P_{\le L}u(r):
r\in J,\ L\in\mathbb N\cup\{\infty\}\bigr\}}^{\,H_x^{1/2}}\subset H_x^{1/2}.
\]
Thanks to the continuity of $u$ in time, the set $\mathcal K$ is compact in $H_x^{1/2}$. We then apply
Proposition \ref{prop 2.2} to the set $\mathcal K$ and denote by $h$ the length number of the time interval given by Proposition \ref{prop 2.2}.

To proceed, we may w.l.o.g. assume $s\le t$ with $t-s\le h$. By calculating $\pt_\tau\Phi_{L,\tau}(f)$ in the way
\begin{align*}
\pt_\tau\Phi_{L,\tau}(f)&=\left.\frac{\partial}{\partial\varepsilon}\Phi_{L,\tau+\varepsilon}(f)\right|_{\varepsilon=0}
=\left.\frac{\partial}{\partial\varepsilon}\left(\Phi_{L,\tau}(\Phi_{L,\vare}(f))\right)\right|_{\varepsilon=0}
=D_f\Phi_{L,\tau}(f)\cdot\left.\partial_t\Phi_{L,t}(f)\right|_{t=0}
\end{align*}
using chain rule, we arrive, in conjunction with \eqref{2.2}, at
\[
\partial_\tau\Phi_{L,\tau}(f)
=D_f\Phi_{L,\tau}(f)
\bigl[
i\Delta f-i\mu P_{\le L}
(|P_{\le L}f|^2P_{\le L}f)
\bigr]
\]
for given initial data $f$. Hence
\begin{equation*}
\begin{aligned}
&\frac{d}{dr}
\Phi_{L,t-r}(P_{\le L}u(r))\\
&\quad=-\pt_\tau \Phi_{L,t-r}(P_{\le L}u(r))+D_f\Phi_{L,t-r}(P_{\le L}u(r))\cdot \pt_r P_{\leq L}u(r)\\
&\quad=D_f\Phi_{L,t-r}(P_{\leq L}u(r))
\bigl[
i\mu P_{\le L}
(|P^2_{\le L}u(r)|^2P^2_{\le L}u(r))
-i\Delta P_{\le L}u(r)\bigr]\\
&\quad\quad+D_f\Phi_{L,t-r}(P_{\le L}u(r))\bigr[i\Delta P_{\leq L}u-i\mu P_{\leq L}(|u|^2 u)\bigr]
\\
&\quad=-i\mu D_f\Phi_{L,t-r}(P_{\le L}u(r))
\Bigl[
P_{\le L}\bigl(
|u(r)|^2u(r)
-|P_{\le L}^2u(r)|^2P_{\le L}^2u(r)
\bigr)
\Bigr].
\end{aligned}
\end{equation*}
Integrating from $s$ to $t$
and using \eqref{2.4}, we conclude that
\begin{equation*}
\begin{aligned}
&\|P_{\le L}u(t)
-\Phi_{L,t-s}(P_{\le L}u(s))\|_{H_x^{-1/2}}\\
&\quad\lesssim \int_s^t
\bigl\|P_{\le L}\bigl(
|u(r)|^2u(r)
-|P_{\le L}^2u(r)|^2P_{\le L}^2u(r)
\bigr)\bigr\|_{H_x^{-1/2}}\,dr.
\end{aligned}
\end{equation*}
Recall from Proposition~\ref{prop 2.1} that $u\in L^3_tL^{9/2}_x(J)$. The dominated convergence theorem therefore yields
$P_{\le L}^2u\to u$ in $L^3_tL^{9/2}_x(J)$ as $L\to\infty$. On the other hand, telescoping and
H\"older give
\[
\begin{aligned}
\||v|^2v-|u|^2u\|_{L^1_tL^{3/2}_x(J)}
\lesssim
\bigl(
\|u\|_{L^3_tL^{9/2}_x(J)}^2
+\|v\|_{L^3_tL^{9/2}_x(J)}^2
\bigr)
\|v-u\|_{L^3_tL^{9/2}_x(J)}.
\end{aligned}
\]
Taking $v=P_{\le L}^2u$ and using the embedding
$L_x^{3/2}\hookrightarrow H_x^{-1/2}$, together
with the uniform (in $L$) boundedness of $P_{\le L}$ on
$H_x^{-1/2}$, we obtain
\begin{equation*}
\begin{aligned}
&\|P_{\le L}u(t)
-\Phi_{L,t-s}(P_{\le L}u(s))\|_{H_x^{-1/2}}\\
&\quad\lesssim 
\bigl\|P_{\le L}\bigl(
|u|^2u-|P_{\le L}^2u|^2P_{\le L}^2u
\bigr)\bigr\|_{L^1_tH^{-1/2}_x(J)}
\longrightarrow0\qquad\text{as $L\to\infty$.}
\end{aligned}
\end{equation*}
Next, using telescoping, \eqref{2.3} and \eqref{2.5} it follows
\begin{align*}
\begin{aligned}
&\|\Phi_{L,t-s}(P_{\le L}u(s))
-\Phi_{t-s}(u(s))\|_{H_x^{1/2}}\\
&\quad \leq \|\Phi_{L,t-s}(P_{\le L}u(s))-\Phi_{L,t-s}(u(s))\|_{H_x^{1/2}}+\|\Phi_{L,t-s}(u(s))
-\Phi_{t-s}(u(s))\|_{H_x^{1/2}}\\
&\quad \lesssim C\|P_{\le L}u(s)-u(s)\|_{H_x^{1/2}}+\sup_{\substack{|\tau|\le h,\,f\in\mathcal K}}
\|\Phi_{L,\tau}(f)-\Phi_\tau(f)\|_{H_x^{1/2}}
\longrightarrow0
\end{aligned}
\end{align*}
as $L\to\infty$. Combing with $P_{\leq L}u(t)\to u(t)$ in $H_{x}^{1/2}\hookrightarrow H_x^{-1/2}$ as $L\to\infty$ we conclude that
\begin{equation*}
u(t)=\Phi_{t-s}(u(s)),
\qquad 0\le t-s\le h
\end{equation*}
in $H_{x}^{-1/2}$. Since both sides belong to $H_x^{1/2}$, the identity also holds in $H_x^{1/2}$. The argument for negative time direction is identical. 

Let now $u,v\in C(I;H_x^{1/2})$ with $u(t_0)=v(t_0)$ for some $t_0\in I$. Consider the set 
\[I':=\{t\in I:\,u(t)=v(t)\}.\]
By the continuity of $u$ and $v$, we may invoke the previous argument to conclude that $I'$ is nonempty and must be open and closed simultaneously. Since $I$ is connected, this implies $I'=I$, as desired.
\end{proof}

\section{Proof of Proposition \ref{prop 2.1}}\label{sec3}
In this section we give the proof of Proposition \ref{prop 2.1}. Indeed, we will show that every distributional solution \(u\in C(I;H_x^{1/2})\) to \eqref{cnls} necessarily belongs to \(L_t^2L_x^6(J)\) for every compact subinterval \(J\subset I\), as stated in the following proposition. 

\begin{proposition}\label{prop 3.1}
Let $I\subset \R$ be a nonempty interval and suppose that $u\in C(I; H^{1/2}(\T^3))$ is a distributional solution of \eqref{cnls}. Then for any compact $J\subset I$ it holds
$
u\in L_t^2 L_x^{6}(J).
$
\end{proposition}

As a consequence, Proposition 2.1 then follows by interpolation.

\begin{proof}[Proof of Proposition \ref{prop 2.1}]
Notice that $u\in L_t^\infty L_x^3(I)$ due to the Sobolev embedding $H_x^{1/2}\hookrightarrow L_x^3$. The desired claim follows from interpolating $L_t^\infty L_x^3$ with $L_t^2 L_x^6$ deduced from Proposition \ref{prop 3.1}.
\end{proof}

Before proceeding with the proof of Proposition \ref{prop 3.1}, we introduce some additional notation. Specifically, we define the low- and high-modulation cutoff operators $\mathcal L$ and $\mathcal H$ trough their space-time Fourier transforms, respectively, as follows:
\begin{align*}
\widetilde{\mathcal{L}z}(\tau,n)
&=\sum_N \phi_N(n)
\left[1-\chi((\tau+|n|^2)/N)\right]
\widetilde z(\tau,n),
\\
\widetilde{\mathcal{H}f}(\tau,n)
&=-\sum_N \phi_N(n)
\frac{\chi((\tau+|n|^2)/N)}{\tau+|n|^2}
\widetilde f(\tau,n).
\end{align*}
Here, the function $\chi\in C^\infty(\R)$ is a real even function which equals zero on $|\tau|\leq 1$ and one on $|\tau|\geq 2$. The proof of Proposition \ref{prop 3.1} relies on the following key lemmas giving crucial estimates for the operators $\mathcal{L}$ and $\mathcal{H}$.

\begin{lemma}[Low-part strong estimate]\label{lem 3.2}
For $f:\R\times\T^3\to\C$ it holds
\begin{equation*}
\|\mathcal{L}f\|_{L_{t}^2 L_x^6(\R)}\lesssim \|f\|_{L_t^2 H_x^{1/2}(\R)}.
\end{equation*}
\end{lemma}

\begin{lemma}[High-part strong estimate]\label{lem 3.3}
For $f:\R\times\T^3\to\C$ it holds
\begin{equation*}
\|\mathcal{H}f\|_{L_{t}^2 L_x^6(\R)}\lesssim \|f\|_{L_t^2 L_x^{6/5}(\R)}.
\end{equation*}
\end{lemma}

\begin{lemma}[High-part weak estimate]\label{lem 3.4}
For $f:\R\times\T^3\to\C$ it holds
\begin{equation*}
\|\mathcal{H}f\|_{L_{t}^\infty L_x^{5,\infty}(\R)}\lesssim \|f\|_{L_t^\infty L_x^{15/13,\infty}(\R)}.
\end{equation*}
\end{lemma}

\begin{lemma}[Fractional product rule]\label{lem 3.5}
For $f,g,h:\T^3\to \C$ it holds
\begin{equation*}
\|(1-\Delta)^{-1/5}(fgh)\|_{L_x^{15/13,\infty}}\lesssim \|f\|_{H_x^{1/2}}\|g\|_{H_x^{1/2}}\|(1-\Delta)^{-1/5}h\|_{L_x^{5,\infty}}.
\end{equation*}
\end{lemma}

With Lemma \ref{lem 3.2} to Lemma \ref{lem 3.5} in place, we give the proof of Proposition \ref{prop 3.1}.

%
%
%

\begin{proof}[Proof of Proposition~\ref{prop 3.1}]
Fix a compact interval $J=[a,b]\subset I$.
We first extend $u$ to the whole time axis by setting
\[
U(t)=
\begin{cases}
e^{i(t-a)\Delta}u(a), & t<a,\\
u(t), & a\le t\le b,\\
e^{i(t-b)\Delta}u(b), & t>b.
\end{cases}
\]
Choose $\eta\in C_c^\infty(\R)$ such that
$\eta=1$ in a neighborhood of $J$, and let
$w=\eta U$. Since $u\in C(I;H_x^{1/2})$,
we have $w\in L_t^2H_x^{1/2}
\cap L_t^\infty H_x^{1/2}$. Moreover, the equation
\[
(i\partial_t+\Delta)w
=\mu 1_J|u|^2w+i\eta'U
\]
holds in the distributional sense. By the definitions
of $\mathcal L$ and $\mathcal H$, we have
$\mathcal H(i\partial_t+\Delta)=I-\mathcal L$.
Consequently,
\begin{equation}\label{eq:w-modulation-identity}
w=\mathcal Lw+
\mathcal H\bigl(\mu1_J|u|^2w+i\eta'U\bigr).
\end{equation}

We next decompose the potential into a low-frequency
part and a small remainder. Set
$R=\|u\|_{L_t^\infty H_x^{1/2}(J)}$.
If $R=0$, the conclusion is immediate.
Otherwise, since $u(J)$ is compact in
$H^{1/2}(\T^3)$, for every $\varepsilon>0$
we may choose $K$ sufficiently large so that
\[
\sup_{t\in J}
\|(I-P_{\le K})u(t)\|_{H_x^{1/2}}
<\varepsilon.
\]
Using the identity
\[
|u|^2-|P_{\le K}u|^2
=(u-P_{\le K}u)\overline u
+(P_{\le K}u)\overline{(u-P_{\le K}u)},
\]
together with the Sobolev embedding
$H^{1/2}(\T^3)\hookrightarrow L^3(\T^3)$,
we obtain
\begin{equation}\label{eq:potential-small}
\bigl\||u|^2-|P_{\le K}u|^2
\bigr\|_{L_t^\infty L_x^{3/2}(J)}
\lesssim \|(I-P_{\leq K})u\|_{L_t^\infty L_x^{3}(J)}\|u\|_{L_t^\infty L_x^{3}(J)}
\lesssim\varepsilon R.
\end{equation}

Define
\[
\mathcal{T}f=\mu\mathcal H
\bigl[1_J(|u|^2-|P_{\le K}u|^2)f\bigr].
\]
We first estimate $\mathcal{T}$ in $L_t^2L_x^6$.
By Lemma~\ref{lem 3.3}, H\"older and \eqref{eq:potential-small}, we have
\begin{equation}\label{eq:T-strong}
\begin{aligned}
\|\mathcal{T}f\|_{L_t^2L_x^6(\R)}
\lesssim
\|1_J(|u|^2-|P_{\le K}u|^2)f
\|_{L_t^2L_x^{6/5}(\R)}
\lesssim
\varepsilon R\|f\|_{L_t^2L_x^6(J)}.
\end{aligned}
\end{equation}
Similarly, using Lemma~\ref{lem 3.4}, Lemma~\ref{lem 3.5} and the fact that $\mathcal H$ commutes
with $(1-\Delta)^{-1/5}$, we obtain
\begin{equation}\label{eq:T-weak}
\begin{aligned}
&\|(1-\Delta)^{-1/5}\mathcal{T}f
\|_{L_t^\infty L_x^{5,\infty}(\R)}\\
&\quad\lesssim
\|(1-\Delta)^{-1/5}
[1_J(|u|^2-|P_{\le K}u|^2)f]
\|_{L_t^\infty L_x^{15/13,\infty}(\R)}\\
&\quad\lesssim
\|(I-P_{\leq K})u\|_{L_t^\infty H_x^{1/2}(J)}
\|u\|_{L_t^\infty H_x^{1/2}(J)}
\|(1-\Delta)^{-1/5}f\|_{L_t^\infty L_x^{5,\infty}(J)}
\\
&\quad\lesssim
\varepsilon R
\|(1-\Delta)^{-1/5}f
\|_{L_t^\infty L_x^{5,\infty}(J)}.
\end{aligned}
\end{equation}
We shall choose $\varepsilon$ sufficiently small so that $C\vare R\leq 1/2$, and in turn \eqref{eq:T-strong} and \eqref{eq:T-weak} define contractions on corresponding function spaces.

Next, by decomposing $|u|^2$ to
$$|u|^2=(|u|^2-|P_{\leq K}u|^2)+|P_{\leq K}u|^2,$$
we may rewrite \eqref{eq:w-modulation-identity} as
\begin{equation}\label{eq:w-fixedpoint}
\begin{aligned}
w=G+\mathcal T w,\qquad
G=\mathcal Lw+
\mathcal H\bigl(
\mu1_J|P_{\le K}u|^2w+i\eta'U
\bigr),
\end{aligned}
\end{equation}
where $G$ and $\mathcal T w$ can be seen as the low and high part of $w$ respectively. We first show that $G\in L_t^2L_x^6$. Indeed, Lemma~\ref{lem 3.2} already gives $\mathcal Lw\in L_t^2L_x^6$. Using Bernstein we also know that $P_{\le K}u\in L_t^\infty L_x^\infty(J)$. Moreover, $w$ and $\eta'U$ are compactly supported in time and therefore belong to $L_t^\infty H_x^{1/2}$. H\"older and Sobolev's embedding $H_x^{1/2}\hookrightarrow L_x^3$ then yield
\[
\mu1_J|P_{\le K}u|^2w+i\eta'U
\in L_t^2L_x^{6/5}.
\]
Applying Lemma~\ref{lem 3.3}, we conclude the claim $G\in L_t^2L_x^6$. We next claim
\begin{align}\label{3.12}
\|(1-\Delta)^{-1/5}w\|_{L_t^\infty L_x^{5,\infty}(J)}+\|(1-\Delta)^{-1/5}G\|_{L_t^\infty L_x^{5,\infty}(J)}<\infty.
\end{align}
Indeed, we first obtain by using Sobolev's embedding that
\[
\|(1-\Delta)^{-1/5}w
\|_{L_t^\infty L_x^{5,\infty}}\lesssim \|(1-\Delta)^{-1/5}w
\|_{L_t^\infty L_x^{5}}
\lesssim
\|w\|_{L_t^\infty H_x^{1/2}}
<\infty.
\]
Since $w=G+\mathcal{T}w$, the estimate
\eqref{eq:T-weak} further implies
\begin{equation*}
\|(1-\Delta)^{-1/5}G
\|_{L_t^\infty L_x^{5,\infty}}\lesssim \|(1-\Delta)^{-1/5}w
\|_{L_t^\infty L_x^{5,\infty}}<\infty,
\end{equation*}
from which \eqref{3.12} follows.

We now consider the iteration
\begin{equation}\label{eq:w-iteration}
w^{0}=G,\qquad
w^{j+1}=G+\mathcal{T}w^{j}.
\end{equation}

By \eqref{eq:T-strong} and Banach fixed point theorem (by choosing $C\vare R\leq 1/2<1$), the sequence
$(w^{j})_j$ converges in $L_t^2L_x^6(J)$ to some $v\in L_t^2L_x^6(J)$ satisfying 
\[v=G+\mathcal{T}v\]
in $L_{t}^2 L_x^6(J)$. The remaining task is to show that $w=v$ in distributional sense, from which it follows $w\in L_{t}^2 L_x^6(J)$ and also the desired claim $u\in L_{t}^2 L_x^6(J)$ (using the fact $u=w$ on $J$).
We now subtract \eqref{eq:w-fixedpoint} from \eqref{eq:w-iteration}, and apply \eqref{eq:T-weak} iteratively from $w^{j}$ down to $w^0$, to obtain
\begin{equation}\label{eq:w-weak-convergence}
\begin{aligned}
\|(1-\Delta)^{-1/5}(w^{j}-w)
\|_{L_t^\infty L_x^{5,\infty}}
\le 2^{-j}
\|(1-\Delta)^{-1/5}(G-w)
\|_{L_t^\infty L_x^{5,\infty}}
\longrightarrow0
\end{aligned}
\end{equation}
as $j\to\infty$. 
Both convergence (\eqref{eq:w-iteration} in the space $L_t^2 L_x^6$ and \eqref{eq:w-weak-convergence}) imply that $v$ and $w$ will indeed coincide in the distributional sense, and the desired proof is complete.
\end{proof}

The remaining part of this section will be devoted to the proofs of Lemma \ref{lem 3.2} to Lemma \ref{lem 3.5}.

\subsection{Low-part strong estimate}
We begin by stating a useful short-time Strichartz estimate introduced in \cite{Burq1} (see the estimate above (3.10) in \cite{Burq1}). In particular, by restricting the time interval to length \(\sim N^{-1}\), we can go beyond the range of exponents covered by the usual Bourgain-type frequency-localized Strichartz estimates, and obtain Strichartz estimates with lower time integrability exponents.

\begin{lemma}[Short time Strichartz on tori, {\cite{Burq1}}]\label{lem:BGT-short-time}
For any dyadic number $N\geq 1$ and time interval $J_N\subset\mathbb R$
with $|J_N|\lesssim N^{-1}$, we have
\begin{equation*}
\|e^{it\Delta}P_N f\|_{L_t^2(J_N;L_x^6(\mathbb T^3))}
\lesssim C\|f\|_{L^2(\mathbb T^3)}.
\end{equation*}
\end{lemma}

We are now ready to give the proof of Lemma~\ref{lem 3.2}.

\begin{proof}[Proof of Lemma~\ref{lem 3.2}]
Let $\psi=1-\chi$. By the definition of $\mathcal L$, we may write
$\mathcal Lf=\sum_N\mathcal L_Nf$, where
\begin{equation}\label{eq:L-kernel}
\mathcal L_Nf(t)
=\int_{\R}N\check\psi(N(t-s))
e^{i(t-s)\Delta}P_Nf(s)\,ds.
\end{equation}
We claim that
\begin{equation}\label{eq:L-dyadic}
\|\mathcal L_Nf\|_{L_t^2L_x^6(\R)}
\lesssim N^{1/2}
\|P_Nf\|_{L_t^2L_x^2(\R)}.
\end{equation}
To prove \eqref{eq:L-dyadic}, partition first the time axis into
intervals
$I_k=[k/N,(k+1)/N)$, $k\in\Z$.
For $t\in I_k$ and $s\in I_\ell$, the rapid
decay of $\check\psi$ gives, for any $A>1$,
\[
N|\check\psi(N(t-s))|
\lesssim_A N(1+|k-\ell|)^{-A}.
\]
Moreover, for each fixed $s\in I_\ell$,
Lemma~\ref{lem:BGT-short-time} yields
\[
\begin{aligned}
\|e^{i(t-s)\Delta}P_Nf(s)
\|_{L_t^2L_x^6(I_k)}
&=
\|e^{it\Delta}e^{-is\Delta}P_Nf(s)
\|_{L_t^2L_x^6(I_k)}\\
&\lesssim
\|e^{-is\Delta}P_Nf(s)\|_{L_x^2}
=\|P_Nf(s)\|_{L_x^2}.
\end{aligned}
\]
Applying then Minkowski's inequality to
\eqref{eq:L-kernel}, we obtain
\[
\begin{aligned}
\|\mathcal L_Nf\|_{L_t^2L_x^6(I_k)}
&\leq
C_AN\sum_{\ell\in\Z}(1+|k-\ell|)^{-A}
\int_{I_\ell}\|P_Nf(s)\|_{L_x^2}\,ds\\
&\leq
C_AN^{1/2}\sum_{\ell\in\Z}
(1+|k-\ell|)^{-A}
\|P_Nf\|_{L_t^2(I_\ell;L_x^2)},
\end{aligned}
\]
where the last inequality follows from Cauchy--Schwarz and $|I_\ell|=N^{-1}$. For large $A$, the kernel $(1+|k|)^{-A}$ belongs to $\ell^1(\Z)$. Hence, by choosing $A\gg 1$, the discrete Young's convolution inequality gives
\[
\begin{aligned}
\|\mathcal L_Nf\|_{L_t^2L_x^6(\R)}^2
=\sum_{k\in\Z}
\|\mathcal L_Nf\|_{L_t^2L_x^6(I_k)}^2
\lesssim_A
N\sum_{\ell\in\Z}
\|P_Nf\|_{L_t^2L_x^2(I_\ell)}^2
=N\|P_Nf\|_{L_t^2L_x^2(\R)}^2,
\end{aligned}
\]
which proves \eqref{eq:L-dyadic}.

It remains to sum over the dyadic frequencies. By fixing some $A\gg 1$ and using Littlewood--Paley inequality and Minkowski, we obtain
\[
\begin{aligned}
\|\mathcal Lf\|_{L_t^2L_x^6(\R)}
&\lesssim
\left\|
\left(\sum_N|\mathcal L_Nf|^2\right)^{1/2}
\right\|_{L_t^2L_x^6(\R)}
\le
\left(
\sum_N\|\mathcal L_Nf\|_{L_t^2L_x^6(\R)}^2
\right)^{1/2}\\
&\lesssim
\left(
\sum_N N\|P_Nf\|_{L_t^2L_x^2(\R)}^2
\right)^{1/2}
\lesssim
\|f\|_{L_t^2H_x^{1/2}(\R)},
\end{aligned}
\]
as desired.
\end{proof}

\subsection{High-part strong estimate}

In this subsection we give the proof of Lemma~\ref{lem 3.3}. Throughout the paper, the function $k$ is defined by 
\[
k=\mathcal F_t^{-1}\left(-\chi(\tau)/\tau\right).
\]
Notice that since the function $\ld\mapsto-\chi(\lambda)/\lambda$ is not of $L^1$-class, the function $k$ exhibits a singularity at zero, causing possible obstruction for our analysis. We will partition $k$ into certain pieces, which admit better decay properties, to overcome such difficulty. 

\begin{lemma}\label{lem:H-time-kernel}
There exist functions $k_j$, $j\geq 0$, such that
\(
k=\sum_{j\geq 0}k_j
\)
with $k_j$ satisfying
\[
\operatorname{supp}k_0\subseteq\{|h|\leq 2\},
\qquad
\operatorname{supp}k_j
\subseteq\{2^{j-1}\leq |h|\leq 2^{j+1}\},
\quad j\geq 1.
\]
For every dyadic $N\geq 1$ and $A\in\N$,
\begin{equation}
\label{eq:H-kernel-L1}
\|k_j\|_{L^\infty(\R)}\lesssim 2^{-Aj},\qquad
\|k_j(N\cdot)\|_{L^1(\R)}
\lesssim_A N^{-1}2^{(1-A)j}.
\end{equation}
For each $j\geq 0$, there also exist
$h_j^+,h_j^-\in\mathcal S(\R)$ such that
\[
h_j^+(h)=k_j(h)\quad(h>0),
\qquad
h_j^-(h)=k_j(h)\quad(h<0)
\]
and
\begin{equation}
\label{eq:H-kernel-Fourier-L1}
\|\widehat{h_j^+}\|_{L^1(\R)}
+\|\widehat{h_j^-}\|_{L^1(\R)}
\lesssim_A 2^{-Aj}.
\end{equation}
\end{lemma}

\begin{proof}
Set
\(
b=\mathcal F_t^{-1}(1-\chi).
\)
By definition, we have 
\(
k'=-i\delta_0+ib
\) holding in the distributional sense. On each of the half-lines $(0,\infty)$ and $(-\infty,0)$,
the function $k$ is therefore smooth and satisfies $k'=ib$.
Since $\widehat k=-\chi(\lambda)/\lambda$ belongs to $L^2(\R)$,
we also have $k\in L^2(\R)$. Consequently,
\[
k(h)=
\begin{cases}
-i\displaystyle\int_h^\infty b(s)\,ds,
& h>0,\\[6pt]
i\displaystyle\int_{-\infty}^h b(s)\,ds,
& h<0.
\end{cases}
\]
Notice that $\widehat{b}\in C_c^\infty(\R)$. It follows from Cauchy--Schwarz that, for every $\ell\geq 0$ and every $A>0$,
\begin{equation}
\label{eq:H-kernel-rapid-decay}
|\partial_h^\ell k(h)|
\lesssim_{\ell,A}(1+|h|)^{-A},
\qquad h\neq 0.
\end{equation}
Choose a smooth dyadic partition of unity
$\sum_{j\geq 0}\eta_j(h)=1$ for $h\neq 0$, with
\begin{align}\label{bound eta}
\operatorname{supp}\eta_0\subseteq\{|h|\leq 2\},
\quad
\operatorname{supp}\eta_j
\subseteq\{2^{j-1}\leq |h|\leq 2^{j+1}\},\quad
|\partial_h^\ell\eta_j(h)|\lesssim_\ell 2^{-j\ell},
\quad j\geq 1
\end{align}
and define $k_j:=\eta_j k$. By
\eqref{eq:H-kernel-rapid-decay}, for any $A>0$, we have
\(
\|k_j\|_{L^\infty(\R)}\lesssim_A 2^{-Aj}, 
\)
which already gives the first estimate in \eqref{eq:H-kernel-L1}. Combining the fact that $|\operatorname{supp}\eta_j|\sim 2^j$ it follows
\[
\|k_j\|_{L^1(\R)}\lesssim 2^{(1-A)j},
\]
which gives \eqref{eq:H-kernel-L1} by a change of variables.

We next prove \eqref{eq:H-kernel-Fourier-L1}. For $j\geq 1$, by noticing that the support of $k_j$ is away from zero, we may set
\[h_j^+(h)=h_j^-(h)=k_j(h).\]
For $j=0$, we simply define
\[
h_0^+(h)=-i\eta_0(h)\int_h^\infty b(s)\,ds,\qquad
h_0^-(h)=i\eta_0(h)\int_{-\infty}^h b(s)\,ds.
\]
By \eqref{eq:H-kernel-rapid-decay} and the derivative bounds for $\eta_j$ given in \eqref{bound eta}, it holds
\(
\|h_j^{\pm}\|_{H^1(\R)}
\lesssim 2^{-Aj}.
\)
Finally, Cauchy--Schwarz and Plancherel give
\[
\|\widehat g\|_{L^1_{\tau}(\R)}\lesssim \|(1+\cdot)^{-1}\|_{L_{\tau}^{2}(\R)}\|(1+\cdot)\widehat g\|_{L_t^2(\R)}
\lesssim \|g\|_{H_t^1(\R)}
\]
for every $g\in H^1(\R)$. Setting $g=h_j^\pm$
proves \eqref{eq:H-kernel-Fourier-L1}.
\end{proof}

\begin{lemma}\label{lem:H-vector-Strichartz}
Let $\{f_N\}_N$ be a finite family of functions on
$\R\times\R^3$. Then
\begin{equation*}
\begin{aligned}
\left\|
\left(
\sum_N
\left|
\int_{s<t}
e^{i(t-s)\Delta_{\R^3}}f_N(s)\,ds
\right|^2
\right)^{1/2}
\right\|_{L_t^2L_x^6(\R)}
\lesssim
\left\|
\left(\sum_N|f_N|^2\right)^{1/2}
\right\|_{L_t^2L_x^{6/5}(\R)}.
\end{aligned}
\end{equation*}
The same estimate holds with the region $s<t$ replaced by $s>t$. 
\end{lemma}

\begin{proof}
We recall the classical Strichartz estimate on $\R^3$, see e.g. Keel and Tao \cite{EndpointStrichartz}:
\begin{equation*}
\left\|
\int_{s<t}e^{i(t-s)\Delta_{\R^3}}f(s)\,ds
\right\|_{L_t^2L_x^6(\R)}
\lesssim
\|f\|_{L_t^2L_x^{6/5}(\R)}.
\end{equation*}
Two applications of Minkowski's inequality then yield
\begin{align*}
&\left\|
\left(
\sum_N
\left|
\int_{s<t}
e^{i(t-s)\Delta_{\R^3}}f_N(s)\,ds
\right|^2
\right)^{1/2}
\right\|_{L_t^2L_x^6(\R)}\\
&\qquad\leq
\left(
\sum_N
\left\|
\int_{s<t}
e^{i(t-s)\Delta_{\R^3}}f_N(s)\,ds
\right\|_{L_t^2L_x^6(\R)}^2
\right)^{1/2}
\\
&\qquad\lesssim
\left(
\sum_N
\|f_N\|_{L_t^2L_x^{6/5}(\R)}^2
\right)^{1/2}
\leq
\left\|
\left(\sum_N|f_N|^2\right)^{1/2}
\right\|_{L_t^2L_x^{6/5}(\R)}.
\end{align*}
The case over $s>t$ follows verbatim, we omit the repeating details here.
\end{proof}

In the following we shall also use $e^{it\Delta_{\R^3}}$ to denote the Schr\"odinger propagators on $\R^3$. In particular, the spatial norm in this case shall also be taken over $\R^3$. The strategy for proving the claims will then be first establishing estimates w.r.t. the spatial domain $\R^3$. At the end, we turn such Euclidean estimates to the ones on $\T^3$ by combining with certain extension arguments.       

For $j\geq 0$ and $f:\R\times\R^3\to\C$, define
\[
\mathcal H_{j,N}^{\R^3}f(t)
=
\int_{\R}
k_j(Nh)e^{ih\Delta_{\R^3}}
P_Nf(t-h)\,dh
\]
and
\[
\mathcal H_j^{\R^3}f
=\sum_N\mathcal H_{j,N}^{\R^3}f.
\]
Lemma \ref{lem:H-time-kernel} and Lemma \ref{lem:H-vector-Strichartz} then yield useful decay mapping properties for the operator $H_j^{\R^3}$, as stated in the following lemma.

\begin{lemma}
\label{lem:H-Euclidean}
We have
\begin{equation}
\label{eq:H-Euclidean-estimate}
\|\mathcal H_j^{\R^3}f\|_{L_t^2L_x^6(\R\times\R^3)}
\lesssim
2^{-Aj}
\|f\|_{L_t^2L_x^{6/5}(\R\times\R^3)}.
\end{equation}
\end{lemma}

\begin{proof}
We first split the time integral into its positive and
negative parts respectively:
\begin{align*}
\mathcal H_{j,N}^{\R^3,+}f(t)&
=
\int_{h>0}
k_j(Nh)e^{ih\Delta_{\R^3}}
P_Nf(t-h)\,dh,\\ 
\mathcal H_{j,N}^{\R^3,-}f(t)&
=
\int_{h<0}
k_j(Nh)e^{ih\Delta_{\R^3}}
P_Nf(t-h)\,dh.
\end{align*}
It suffices to consider the positive part $H_{j,N}^{\R^3,+}F(t)$, the negative part $H_{j,N}^{\R^3,-}F(t)$ can be dealt similarly with. The change of variables $s=t-h$ and the identity $k_j(h)=h_j^+(h)$ for $h>0$ give
\[
\mathcal H_{j,N}^{\R^3,+}f(t)
=
\int_{s<t}
h_j^+(N(t-s))
e^{i(t-s)\Delta_{\R^3}}
{P_N}f(s)\,ds.
\]
By Fourier inversion,
\begin{align*}
\begin{aligned}
\mathcal H_{j,N}^{\R^3,+}f(t)
&=
\frac{1}{2\pi}
\int_{\R}
\widehat{h_j^+}(\nu)e^{i\nu Nt}
\\
&\quad\times
\left(
\int_{s<t}
e^{i(t-s)\Delta_{\R^3}}
e^{-i\nu Ns}{P_N}f(s)\,ds
\right)d\nu.
\end{aligned}
\end{align*}
Now Minkowski, Lemma~\ref{lem:H-vector-Strichartz},
\eqref{eq:H-kernel-Fourier-L1} and the Littlewood-Paley inequality imply
\begin{align*}
\left\|
\left(
\sum_N
|\mathcal H_{j,N}^{\R^3,+}f|^2
\right)^{1/2}
\right\|_{L_t^2L_x^6(\R\times\R^3)}
&\lesssim
\|\widehat{h_j^+}\|_{L^1(\R)}
\left\|
\left(
\sum_N|{P_N}f|^2
\right)^{1/2}
\right\|_{L_t^2L_x^{6/5}(\R\times\R^3)}
\\
&\lesssim
2^{-Aj}
\left\|
\left(
\sum_N|{P_N}f|^2
\right)^{1/2}
\right\|_{L_t^2L_x^{6/5}(\R\times\R^3)}\\
&\lesssim 2^{-Aj}\|f\|_{L_t^2L_x^{6/5}(\R\times\R^3)},
\end{align*}
as claimed.
\end{proof}

By suitable extension argument we may transform Lemma \ref{lem:H-Euclidean} to a periodic version, as stated in the following lemma.

\begin{lemma}
\label{lem 3.10}
Let $\Omega=[-\pi,\pi]^3$ and let $F_{\mathrm{per}}$ denote the $2\pi$-periodic extension of a function  $F:\T^3\to \C$. For $j\geq 0$ and $f:\R\times\T^3\to\C$, define
\begin{align*}
\mathcal H_j^{\T^3}f(t)
&:=\sum_N\int_{\R}k_j(Nh)
e^{ih\Delta_{\T^3}}P_N f(t-h)\,dh.
\end{align*}
Choose $\rho_j\in C_c^\infty(\R^3)$ such that
$0\leq\rho_j\leq1$ and
\[
\rho_j(x)=1
\quad\text{for }|x|_\infty\leq C_*2^j,
\qquad
\rho_j(x)=0
\quad\text{for }|x|_\infty\geq 2C_*2^j,
\]
where $C_*>0$ is a sufficiently large fixed constant. On $\R\times\Omega$, also define the remainder operator $R_j^{\T^3}$ via
\begin{equation}
\mathcal H_j^{\T^3}f
=
\mathcal H_j^{\R^3}(\rho_jf_{\mathrm{per}})
+\mathcal R_j^{\T^3}f.
\label{eq:exact-periodization}
\end{equation}
Then for every $1\leq q\leq \infty$, $1<p<\infty$, $L>3$, $A>0$ we have
\begin{align}
\|\rho_jf_{\mathrm{per}}\|_{L_t^qL_x^p(\R\times\R^3)}
&\lesssim 2^{3j/p}
\|f\|_{L_t^qL_x^p(\R\times\T^3)},
\label{eq:extension-strong}\\
\|\rho_jf_{\mathrm{per}}\|_{L_t^q L_x^{p,\infty}(\R\times\R^3)}
&\lesssim 2^{3j/p}
\|f\|_{L_t^q L_x^{p,\infty}(\R\times\T^3)},
\label{eq:extension-weak}\\
\|\mathcal R_j^{\T^3}f\|_{L_t^qL_x^\infty(\R\times\Omega)}
&\lesssim_{A,L,p}
2^{(4-L-A)j}
\|f\|_{L_t^qL_x^p(\R\times\T^3)},
\label{eq:remainder-strong}\\
\|\mathcal R_j^{\T^3}f\|_{L_t^q L_x^\infty(\R\times\Omega)}
&\lesssim_{A,L,p}
2^{(4-L-A)j}
\|f\|_{L_t^q L_x^{p,\infty}(\R\times\T^3)}.
\label{eq:remainder-weak}
\end{align}
\end{lemma}

\begin{remark}
We will only use strong estimates \eqref{eq:extension-strong} and \eqref{eq:remainder-strong} in the proof of Lemma \ref{lem 3.3}, while \eqref{eq:extension-weak} and \eqref{eq:remainder-weak} will be used in the proof of Lemma \ref{lem 3.4}.
\end{remark}

\begin{proof}
We first establish
\eqref{eq:extension-strong} and
\eqref{eq:extension-weak}.
Since $\rho_j$ is supported in a box of
side length $O(2^j)$, its support intersects
at most $O(2^{3j})$ translates of $\Omega$.
By periodicity, for every fixed $t$,
\[
\|\rho_jF_{\mathrm{per}}(t)\|_{L_x^p(\R^3)}^p
\lesssim
2^{3j}\|F(t)\|_{L_x^p(\T^3)}^p.
\]
Taking the $L_t^q$ norm yields
\eqref{eq:extension-strong}. To prove the weak estimate, observe that
for every $\lambda>0$,
\[
\begin{aligned}
\bigl|\{x\in\R^3:
|\rho_j(x)F_{\mathrm{per}}(t,x)|>\lambda\}\bigr|
\lesssim
2^{3j}
\bigl|\{x\in\Omega:
|F(t,x)|>\lambda\}\bigr|.
\end{aligned}
\]
Taking the supremum over $\lambda>0$
in the definition of the weak $L^p$ norm
gives
\[
\|\rho_jF_{\mathrm{per}}(t)\|_{L_x^{p,\infty}(\R^3)}
\lesssim
2^{3j/p}
\|F(t)\|_{L_x^{p,\infty}(\T^3)}.
\]
Again, taking the $L_t^q$ norm proves \eqref{eq:extension-weak}.

To prove \eqref{eq:remainder-strong} and \eqref{eq:remainder-weak}, notice first that it suffices to only prove \eqref{eq:remainder-weak} using the embedding $L^p\hookrightarrow L^{p,\infty}$. To proceed, first write down an explicit formula for $\mathcal R_j^{\T^3}f$ via \eqref{eq:exact-periodization}. For each dyadic $N\geq1$, denote by
\begin{align}\label{kn kernel}
K_N^{\R^3}(h,z)
:=
\frac{1}{(2\pi)^3}
\int_{\R^3}
e^{iz\cdot\xi-ih|\xi|^2}
\phi_N(\xi)\,d\xi
\end{align}
the kernel of
$e^{ih\Delta_{\R^3}}P_N^{\R^3}$. Poisson summation therefore gives
\begin{equation}
\sum_{\ell\in\Z^3}
K_N^{\R^3}(h,z+2\pi\ell)
=
\frac{1}{(2\pi)^3}
\sum_{n\in\Z^3}
e^{in\cdot z-ih|n|^2}\phi_N(n).
\label{eq:Poisson-periodization}
\end{equation}
Now using \eqref{eq:Poisson-periodization} and the definition of $e^{it\Delta_{\T^3}}P_N$, we obtain, for $x\in\Omega$, that
\begin{align}
\begin{aligned}
&e^{ih\Delta_{\T^3}}P_N f(t-h,x)=
\sum_{n\in\Z^3}
e^{in\cdot x-ih|n|^2}\phi_N(n)\widehat{f}(t-h,n).
\\
&\quad=
\int_\Omega
\sum_{\ell\in\Z^3}
K_N^{\R^3}(h,x-y+2\pi\ell)
f(t-h,y)\,dy
\\
&\quad=
\int_{\R^3}
K_N^{\R^3}(h,x-y)
f_{\mathrm{per}}(t-h,y)\,dy.
\end{aligned}
\label{eq:periodic-kernel-identity}
\end{align}
By decomposing 
\[
f_{\mathrm{per}}
=
\rho_jf_{\mathrm{per}}
+(1-\rho_j)f_{\mathrm{per}}
\]
in \eqref{eq:periodic-kernel-identity} we obtain an explicit formula for $R_j^{\T^3} f$, stated as
\begin{equation}
\begin{aligned}
\mathcal R_j^{\T^3}f(t,x)
&=
\sum_N\int_{\R}k_j(Nh)
\int_{\R^3}
K_N^{\R^3}(h,x-y)
(1-\rho_j(y))
f_{\mathrm{per}}(t-h,y)\,dy\,dh\\
&
=
\sum_N\int_{\R}k_j(Nh)
\sum_{\ell\in\Z^3}\int_{\Omega+2\pi\ell}
K_N^{\R^3}(h,x-y)
(1-\rho_j(y))
f_{\mathrm{per}}(t-h,y)\,dy\,dh.
\end{aligned}
\label{eq:periodization-remainder}
\end{equation}
It remains to prove \eqref{eq:remainder-strong} and \eqref{eq:remainder-weak} by making use of \eqref{eq:periodization-remainder}. First recall that the support properties of $k_j$ imply
\[
k_j(Nh)\neq0
\quad\Longrightarrow\quad
|h|N\lesssim 2^j.
\]
On the other hand, if $x\in\Omega$ and $1-\rho_j(y)\neq0$, then
\[
|x-y|\gtrsim C_*2^j,
\]
provided $C_*$ is chosen sufficiently large. For any $L\in\N$, repeated integration by parts in $\xi$ in \eqref{kn kernel} gives 
\begin{equation*}
|K_N^{\R^3}(h,z)|
\lesssim_L N^{3-L}|z|^{-L}
\end{equation*}
whenever $k_j(Nh)\neq0$, $x\in\Omega$, and $1-\rho_j(y)\neq0$. Setting $z=x-y$ with $x\in\Omega$ and $1-\rho_j(y)\neq0$ it then follows
\begin{equation*}
|K_N^{\R^3}(h,x-y)|
\lesssim_L N^{3-L}|x-y|^{-L}.
\end{equation*}
For $L>3$, summing the spatial decay over those fundamental cells which
intersect the support of $1-\rho_j$
gives, uniformly for $x\in\Omega$,
\[
\sum_{\ell\in\Z^3}
\sup_{\substack{y\in\Omega+2\pi\ell\\
1-\rho_j(y)\neq0}}
|x-y|^{-L}
\lesssim_L
2^{(3-L)j}.
\]
Next, partition the $y$-integration in \eqref{eq:periodization-remainder} into the translates
$\Omega+2\pi\ell$, $\ell\in\Z^3$. By periodicity and H\"older in weak $L^p$-norm, stated as 
\[\|F\|_{L_x^1}\lesssim \|F\|_{L_x^{p,\infty}}\|1\|_{L_x^{p',1}}\lesssim \|F\|_{L_x^{p,\infty}},\]
we have for $1<p\leq \infty$
\[
\int_{\Omega+2\pi\ell}
|f_{\mathrm{per}}(t-h,y)|\,dy
\lesssim_p
\|f(t-h)\|_{L_x^{p,\infty}(\T^3)}.
\]
Summing up, we conclude that
\begin{equation}
\begin{aligned}
|\mathcal R_j^{\T^3}Ff(t,x)|
\lesssim_L 
2^{(3-L)j}
\sum_N N^{3-L}
\int_{\R}|k_j(Nh)|
\|f(t-h)\|_{L_x^{p,\infty}(\T^3)}\,dh.
\end{aligned}
\label{eq:remainder-pointwise}
\end{equation}
Applying Minkowski, Young's convolution inequality
in time and also \eqref{eq:H-kernel-L1} to \eqref{eq:remainder-pointwise},
we obtain
\begin{align*}
\|\mathcal R_j^{\T^3}f
\|_{L_t^qL_x^\infty(\Omega)}
&\lesssim_L
2^{(3-L)j}
\sum_N N^{3-L}
\|k_j(N\cdot)\|_{L^1(\R)}
\|f\|_{L_t^qL_x^{p,\infty}(\T^3)}\\
&\lesssim_{A,L}
2^{(4-L-A)j}
\left(\sum_N N^{2-L}\right)
\|f\|_{L_t^qL_x^{p,\infty}(\T^3)}
\lesssim
2^{(4-L-A)j}\|f\|_{L_t^qL_x^{p,\infty}(\T^3)},
\end{align*}
as desired.
\end{proof}

\begin{proof}[Proof of Lemma~\ref{lem 3.3}]
Recall that $k=\sum_{j\geq0}k_j$. By the definitions
of $\mathcal H$ and $\mathcal H_j^{\T^3}$, we have
\[
\mathcal Hf=\sum_{j\geq0}\mathcal H_j^{\T^3}f.
\]
Using the identity
\eqref{eq:exact-periodization} and triangular inequality, it holds
\begin{align*}
\|\mathcal H_j^{\T^3}f\|_{L_t^2L_x^6(\R\times\T^3)}
\lesssim
\|\mathcal H_j^{\R^3}(\rho_jf_{\mathrm{per}})\|
_{L_t^2L_x^6(\R\times\R^3)}
+
\|\mathcal R_j^{\T^3}f\|
_{L_t^2L_x^\infty(\R\times\Omega)}.
\end{align*}
For the first term, applying
\eqref{eq:H-Euclidean-estimate} and
\eqref{eq:extension-strong} with $p=6/5$ gives
\begin{align*}
\|\mathcal H_j^{\R^3}(\rho_jf_{\mathrm{per}})\|
_{L_t^2L_x^6(\R\times\R^3)}
\lesssim 
2^{(-A+5/2)j}
\|f\|_{L_t^2L_x^{6/5}(\R\times\T^3)}.
\end{align*}
For the remainder, we apply
\eqref{eq:remainder-strong} with
$p=6/5$ and $q=2$ to obtain
\[
\|\mathcal R_j^{\T^3}f\|
_{L_t^2L_x^\infty(\R\times\Omega)}
\lesssim
2^{(4-L-A)j}
\|f\|_{L_t^2L_x^{6/5}(\R\times\T^3)}.
\]
By choosing $L$ and $A$ sufficiently large, this in turn implies
\[
\|\mathcal H_j^{\T^3}f\|_{L_t^2L_x^6(\R\times\T^3)}
\lesssim
2^{-A'j}
\|f\|_{L_t^2L_x^{6/5}(\R\times\T^3)}.
\]
for some $A'>0$. Summing over $j$ yields the desired claim.
\end{proof}

\subsection{High-part weak estimate}
In this subsection we give the proof of Lemma \ref{lem 3.4}. We begin with a global space-time estimate. 

\begin{lemma}
\label{lem:global-free-Lorentz}
For $g:\R\times\R^3\to\C$ we have
\begin{equation}
\label{eq:global-free-Lorentz}
\|e^{it\Delta_{\R^3}}g\|_{L_t^1L_x^{15/2,1}(\R\times\R^3)}
\lesssim
\|g\|_{L^{5/4,1}(\R^3)}.
\end{equation}
\end{lemma}

\begin{proof}
Recall first the real interpolation $(L^r,L^\infty)_{1-r/p,1}=L^{p,1}$:
\begin{equation}
\label{eq:Lorentz-interpolation-free}
\|F\|_{L^{p,1}(\R^3)}
\lesssim_{p,r}
\|F\|_{L^r(\R^3)}^{r/p}
\|F\|_{L^\infty(\R^3)}^{1-r/p},
\qquad 1\leq r<p<\infty.
\end{equation}
Fix $T>0$. For $0<t<T$, applying
\eqref{eq:Lorentz-interpolation-free}
with $p=15/2$ and $r=6$, followed by
H\"older in time and the Strichartz estimate
\(\|e^{it\Delta_{\R^3}}g\|_{L_t^2 L_x^6(\R\times\R^3)}\lesssim \|g\|_{L_x^2},\)
yields
\begin{align}\label{3.39}
\begin{aligned}
&\int_0^T
\|e^{it\Delta_{\R^3}}g\|_{L_x^{15/2,1}(\R^3)}
\,dt\\
&\quad\lesssim
\|g\|_{L^1}^{1/5}
\int_0^T
t^{-3/10}
\|e^{it\Delta_{\R^3}}g\|_{L_x^6}^{4/5}\,dt\\
&\quad\lesssim
\|g\|_{L^1}^{1/5}
\left(
\int_0^T
\|e^{it\Delta_{\R^3}}g\|_{L_x^6}^2\,dt
\right)^{2/5}
\left(
\int_0^T t^{-1/2}\,dt
\right)^{3/5}\\
&\quad\lesssim
T^{3/10}
\|g\|_{L^1}^{1/5}
\|g\|_{L^2}^{4/5}.
\end{aligned}
\end{align}
For $t>T$, applying
\eqref{eq:Lorentz-interpolation-free}
with $p=15/2$ and $r=2$, and using the dispersive estimate of $e^{it\Delta_{\R^3}}$, we obtain
\begin{align}
\int_T^\infty
\|e^{it\Delta_{\R^3}}g\|_{L_x^{15/2,1}(\R^3)}
\,dt
\lesssim
\|g\|_{L^1}^{11/15}
\|g\|_{L^2}^{4/15}
\int_T^\infty t^{-11/10}\,dt
\lesssim
T^{-1/10}
\|g\|_{L^1}^{11/15}
\|g\|_{L^2}^{4/15}.
\label{3.40}
\end{align}
Balancing \eqref{3.39} and \eqref{3.40} gives
\begin{align}\label{3.41}
\int_0^\infty
\|e^{it\Delta_{\R^3}}g\|_{L_x^{15/2,1}}
\,dt
\lesssim
\|g\|_{L^1}^{3/5}
\|g\|_{L^2}^{2/5}.
\end{align}
The same estimate for negative times follows from
\(
e^{-it\Delta_{\R^3}}g
=
\overline{e^{it\Delta_{\R^3}}\overline g}.
\)
Since the map $g\mapsto e^{it\Delta_{\R^3}}g$ is linear, the $J$-method of real interpolation
(see e.g. \cite{BerghLofstrom}) allows us to extend \eqref{3.41} to $L_x^{5/4,1}$ due to the real interpolation
\(
(L^1,L^2)_{2/5,1}=L^{5/4,1}.
\)
This proves \eqref{eq:global-free-Lorentz}.
\end{proof}

We next establish a uniform multiplier estimate in $L^{15/2,1}(\R^3)$.

\begin{lemma}
\label{lem:H-Lorentz-multiplier}
For $j\geq0$ and $h\neq0$, define
\[
m_{j,h}(\xi)
=
\sum_N \phi_N(\xi)k_j(Nh).
\]
Then, for any $A>0$, we have, uniformly in $j$ and $h$, that
\begin{equation}
\label{eq:H-multiplier-derivatives}
|\partial_\xi^\gamma m_{j,h}(\xi)|
\lesssim_{A,\gamma}
2^{-Aj}|\xi|^{-|\gamma|}
\qquad \xi\neq0,
\end{equation}
for every multi-index $\gamma$. Consequently,
\begin{equation}
\label{eq:H-Lorentz-multiplier}
\begin{aligned}
\|m_{j,h}(-i\nabla)g\|_{L^{15/2,1}(\R^3)}
+\|m_{j,h}(-i\nabla)^*g\|_{L^{15/2,1}(\R^3)}
\lesssim
2^{-Aj}\|g\|_{L^{15/2,1}(\R^3)}.
\end{aligned}
\end{equation}
\end{lemma}

\begin{proof}
From \eqref{eq:H-kernel-L1} we have
\(
\|k_j\|_{L^\infty(\R)}
\lesssim_A 2^{-Aj}.
\)
Notice also that for each fixed $\xi\neq0$, only a bounded number of dyadic multipliers $\phi_N(\xi)$ can be nonzero. Combining the uniform bound
\[
|\partial_\xi^\gamma\phi_N(\xi)|
\lesssim_\gamma N^{-|\gamma|}\sim |\xi|^{-|\gamma|},
\] 
we obtain
\[
\begin{aligned}
|\partial_\xi^\gamma m_{j,h}(\xi)|
\leq
\sum_N
|k_j(Nh)|
|\partial_\xi^\gamma\phi_N(\xi)|
\lesssim_\gamma
2^{-Aj}|\xi|^{-|\gamma|},
\end{aligned}
\]
which proves \eqref{eq:H-multiplier-derivatives}. Now by the Mikhlin multiplier theorem,
$m_{j,h}(-i\nabla)$ is bounded, say, on $L^7(\R^3)$ and $L^8(\R^3)$, with operator norms bounded by $C2^{-Aj}$.
The real interpolation
\(
(L^7,L^8)_{8/15,1}=L^{15/2,1}
\)
then gives
\[
\|m_{j,h}(-i\nabla)g\|_{L^{15/2,1}}
\lesssim
2^{-Aj}\|g\|_{L^{15/2,1}}.
\]
Since the adjoint multiplier has symbol
$\overline{m_{j,h}(\xi)}$, which satisfies
the same derivative bounds, the corresponding
estimate for $m_{j,h}(-i\nabla)^*$ follows
identically.
\end{proof}

\begin{lemma}
\label{lem:H-Euclidean-weak}
For every $j\geq0$, we have
\begin{equation}
\label{eq:H-Euclidean-weak}
\|\mathcal H_j^{\R^3}f\|
_{L_t^\infty L_x^{5,\infty}(\R\times\R^3)}
\lesssim
2^{-Aj}
\|f\|_{L_t^\infty L_x^{15/13,\infty}(\R\times\R^3)}.
\end{equation}
\end{lemma}

\begin{proof}
By the definition of $\mathcal H_j^{\R^3}$ and the multiplier $m_{j,h}$ introduced above, we have
\[
\mathcal H_j^{\R^3}f(t)
=
\int_{\R}
e^{ih\Delta_{\R^3}}
m_{j,h}(-i\nabla)f(t-h)\,dh.
\]
Fix $t\in\R$ and a trial function $g\in\mathcal S(\R^3)$. Duality then gives
\begin{align*}
|\langle\mathcal H_j^{\R^3}f(t),g\rangle|
\leq
\int_{\R}
\left|
\left\langle
f(t-h),
m_{j,h}(-i\nabla)^*
e^{-ih\Delta_{\R^3}}g
\right\rangle
\right|\,dh.
\end{align*}
Combining the H\"older inequality 
$\|fg\|_{L^1}\lesssim\|f\|_{L^{p,1}}\|g\|_{p',\infty}$ 
for $1<p\leq \infty$ and \eqref{eq:H-Lorentz-multiplier}, it follows
\begin{align*}
|\langle\mathcal H_j^{\R^3}f(t),g\rangle|
&\lesssim
\int_{\R}
\|f(t-h)\|_{L^{15/13,\infty}}
\|m_{j,h}(-i\nabla)^*
e^{-ih\Delta_{\R^3}}g\|_{L^{15/2,1}}
\,dh\\
&\lesssim
2^{-Aj}
\|f\|_{L_t^\infty L_x^{15/13,\infty}}
\int_{\R}
\|e^{-ih\Delta_{\R^3}}g\|_{L^{15/2,1}}
\,dh.
\end{align*}
Applying Lemma~\ref{lem:global-free-Lorentz}
to the last integral yields
\[
|\langle\mathcal H_j^{\R^3}f(t),g\rangle|
\lesssim
2^{-Aj}
\|f\|_{L_t^\infty L_x^{15/13,\infty}}
\|g\|_{L^{5/4,1}}
\]
and the desired claim follows now by duality.
\end{proof}

We are now ready to give the proof of Lemma~\ref{lem 3.4}.

\begin{proof}[Proof of Lemma~\ref{lem 3.4}]
This follows the same line as in the proof of Lemma \ref{lem 3.3}, where we shall simply replace \eqref{eq:H-Euclidean-estimate}, \eqref{eq:extension-strong} and \eqref{eq:remainder-strong} applied therein, by \eqref{eq:H-Euclidean-weak}, \eqref{eq:extension-weak} and \eqref{eq:remainder-weak}. We thus omit the repeating details.
\end{proof}

\subsection{Fractional product rule}
In this subsection we give the proof of Lemma \ref{lem 3.5}. We begin with stating several useful Lorentz-Sobolev embedding. 

\begin{lemma}
\label{lem:Sobolev-Lorentz}
For $f:\T^3\to \C$ it holds 
\begin{align}
\|f\|_{L^{3,2}_x}
+
\|(1-\Delta)^{1/5}f\|_{L^{15/7,2}_x}
&\lesssim \|f\|_{H^{1/2}_x},
\label{eq:Sobolev-Lorentz-particular}\\
\|(1-\Delta)^{-1/5}f\|_{L^{\infty}_x}
&\lesssim \|f\|_{L_x^{15/2,1}}.\label{eq:Sobolev-Lorentz-particular2}
\end{align}
\end{lemma}

\begin{proof}
\eqref{eq:Sobolev-Lorentz-particular} follows directly from real interpolating the standard Sobolev inequalities. 
\eqref{eq:Sobolev-Lorentz-particular2} follows from the embedding \(
H_{p_1,p_2}^{n/p_1}(\mathbb R^n)
\hookrightarrow L_{q_1,q_2}(\mathbb R^n) 
\) by setting $n=3$, $p_1=15/2$, $p_2=1$, $q_1=q_2=\infty$, see e.g. \cite{MachiharaOzawaWadade}.
\end{proof}

We also record some useful fractional Leibniz rules given in term of Lorentz spaces.

\begin{lemma}
\label{lem:bilinear-Lorentz}
For $f,g:\T^3\to\C$ we have
\begin{equation}
\begin{aligned}
\|(1-\Delta)^{1/5}(fg)\|_{L_x^{5/4,1}}
\lesssim
\|f\|_{H^{1/2}}\|g\|_{H^{1/2}}.
\end{aligned}
\label{eq:bilinear-Lorentz}
\end{equation}
\end{lemma}

\begin{proof}
We first recall the following fractional Lorentz--Leibniz rule
\begin{equation*}
\begin{aligned}
\|(1-\Delta)^{1/5}(fg)\|_{L_x^{5/4,1}}
\lesssim
\|(1-\Delta)^{1/5}f\|_{L_x^{15/7,2}}
\|g\|_{L_x^{3,2}}+
\|f\|_{L_x^{3,2}}
\|(1-\Delta)^{1/5}g\|_{L_x^{15/7,2}}.
\end{aligned}
\end{equation*}
from \cite[Theorem 6.1]{CruzUribeNaibo}. Now \eqref{eq:bilinear-Lorentz} follows by combining \eqref{eq:Sobolev-Lorentz-particular}.
\end{proof}

\begin{lemma}
\label{lem:Lorentz-Leibniz}
For $f,g:\T^3\to\C$ we have
\begin{equation*}
\begin{aligned}
\|(1-\Delta)^{1/5}(fg)\|_{L_x^{5/4,1}}
\lesssim
\|(1-\Delta)^{1/5}f\|_{L_x^{5/4,1}}
\left(
\|g\|_{L_x^\infty}
+
\|(1-\Delta)^{1/5}g\|_{L_x^{15/2}}
\right).
\end{aligned}
\end{equation*}
\end{lemma}

\begin{proof}
Let $1<q<15/2$, and define $q_*$ by
\[
\frac1{q_*}=\frac1q-\frac2{15}.
\]
By the standard fractional Leibniz rule in Lebesgue spaces, we have
\[
\begin{aligned}
\|(1-\Delta)^{1/5}(fg)\|_{L_x^q}
\lesssim
\|(1-\Delta)^{1/5}f\|_{L_x^q}
\|g\|_{L_x^\infty}+
\|f\|_{L_x^{q_*}}
\|(1-\Delta)^{1/5}g\|_{L_x^{15/2}}.
\end{aligned}
\]
Since
\(
\frac1{q_*}=\frac1q-\frac{2/5}{3},
\)
the Sobolev embedding yields
\(
\|f\|_{L_x^{q_*}}
\lesssim
\|(1-\Delta)^{1/5}f\|_{L_x^q}.
\)
Consequently,
\begin{equation}
\begin{aligned}
\|(1-\Delta)^{1/5}(fg)\|_{L_x^q}
\lesssim
\|(1-\Delta)^{1/5}f\|_{L_x^q}
\left(
\|g\|_{L_x^\infty}
+
\|(1-\Delta)^{1/5}g\|_{L_x^{15/2}}
\right).
\end{aligned}
\label{eq:strong-fractional-Leibniz}
\end{equation}

Now fix $g$ and consider the linear operator
\[
u\longmapsto
(1-\Delta)^{1/5}
\left[g(1-\Delta)^{-1/5}u\right].
\]
It follows from \eqref{eq:strong-fractional-Leibniz}
that this operator is bounded on $L_x^q$ for every
$1<q<15/2$, with operator norm controlled by
\[
\|g\|_{L_x^\infty}
+
\|(1-\Delta)^{1/5}g\|_{L_x^{15/2}}.
\]
Take $q=6/5$ and $q=4/3$.
The real interpolation 
\(
(L_x^{6/5},L_x^{4/3})_{2/5,1}
=L_x^{5/4,1}
\)
then yields
\[
\begin{aligned}
\left\|
(1-\Delta)^{1/5}
\left[g(1-\Delta)^{-1/5}u\right]
\right\|_{L_x^{5/4,1}}
\lesssim
\left(
\|g\|_{L_x^\infty}
+
\|(1-\Delta)^{1/5}g\|_{L_x^{15/2}}
\right)
\|u\|_{L_x^{5/4,1}}.
\end{aligned}
\]
Taking $u=(1-\Delta)^{1/5}f$
completes the proof.
\end{proof}

We are now ready to give the proof of Lemma \ref{lem 3.5}.

\begin{proof}[Proof of Lemma~\ref{lem 3.5}]
Let $\varphi\in C^\infty(\T^3)$ be a trial function. By duality and
the self-adjointness of $(1-\Delta)^{-1/5}$,
we have
\begin{align*}
\left|
\int_{\T^3}
(1-\Delta)^{-1/5}(fgh)\,\varphi\,dx
\right|
&=
\left|
\int_{\T^3}
(1-\Delta)^{-1/5}h\,
(1-\Delta)^{1/5}
\bigl[fg(1-\Delta)^{-1/5}\varphi\bigr]
\,dx
\right|\\
&\lesssim
\|(1-\Delta)^{-1/5}h\|_{L_x^{5,\infty}}
\left\|
(1-\Delta)^{1/5}
\bigl[fg(1-\Delta)^{-1/5}\varphi\bigr]
\right\|_{L_x^{5/4,1}},
\end{align*}
where we used the Lorentz H\"older inequality. Applying next Lemma~\ref{lem:Lorentz-Leibniz}, we obtain
\begin{align*}
\left\|
(1-\Delta)^{1/5}
\bigl[fg(1-\Delta)^{-1/5}\varphi\bigr]
\right\|_{L_x^{5/4,1}}
&\lesssim
\|(1-\Delta)^{1/5}(fg)\|_{L_x^{5/4,1}}
\left(
\|(1-\Delta)^{-1/5}\varphi\|_{L_x^\infty}
+
\|\varphi\|_{L_x^{15/2}}
\right)\\
&\lesssim
\|(1-\Delta)^{1/5}(fg)\|_{L_x^{5/4,1}}
\|\varphi\|_{L_x^{15/2,1}},
\end{align*}
where the last inequality follows from
\eqref{eq:Sobolev-Lorentz-particular2}
and the embedding
$L_x^{15/2,1}\hookrightarrow L_x^{15/2}$. Now Lemma~\ref{lem:bilinear-Lorentz} yields
\[
\|(1-\Delta)^{1/5}(fg)\|_{L_x^{5/4,1}}
\lesssim
\|f\|_{H_x^{1/2}}\|g\|_{H_x^{1/2}}.
\]
Consequently,
\begin{align*}
\left|
\int_{\T^3}
(1-\Delta)^{-1/5}(fgh)\,\varphi\,dx
\right|
\lesssim
\|f\|_{H_x^{1/2}}
\|g\|_{H_x^{1/2}}
\|(1-\Delta)^{-1/5}h\|_{L_x^{5,\infty}}
\|\varphi\|_{L_x^{15/2,1}}.
\end{align*}
Taking the supremum over
$\|\varphi\|_{L_x^{15/2,1}}\leq1$
and using
$(L_x^{15/2,1})^*=L_x^{15/13,\infty}$,
we conclude that
\[
\|(1-\Delta)^{-1/5}(fgh)\|_{L_x^{15/13,\infty}}
\lesssim
\|f\|_{H_x^{1/2}}
\|g\|_{H_x^{1/2}}
\|(1-\Delta)^{-1/5}h\|_{L_x^{5,\infty}}.
\]
This completes the proof.
\end{proof}

\section{Proof of Proposition \ref{prop 2.2}}\label{sec4}
We close the proof of Theorem \ref{main thm} by giving the proof of Proposition \ref{prop 2.2} in the following.

\begin{proof}[Proof of Proposition~\ref{prop 2.2}]
Up to \eqref{2.4}, the claims of Proposition~\ref{prop 2.2} follow from the classical local well-posedness theory for \eqref{cnls}, we refer e.g. to \cite{HerrTataruTz1,LeeNonalg19} for standard proofs. It suffices therefore to prove \eqref{2.4}.

For each $L\in\mathbb N$, the truncated equation
\eqref{2.2} is a Hamiltonian system with Hamiltonian
\[
E_L(z)
=\frac12\int_{\T^3}|\nabla z|^2\,dx
+\frac{\mu}{4}\int_{\T^3}|P_{\le L}z|^4\,dx.
\]
Consequently, its flow preserves the canonical symplectic form
\[
\omega(g_1,g_2)
=\operatorname{Im}\int_{\T^3}
g_1\overline{g_2}\,dx,
\]
and for any $g_1,g_2\in H_x^{1/2}$, we have 
\[
\omega\bigl(
D_f\Phi_{L,t}(f)g_1,
D_f\Phi_{L,t}(f)g_2
\bigr)
=\omega(g_1,g_2),
\]
see e.g. Tao \cite{TaoBook}. Since $D_f\Phi_{L,t}(f)$ is invertible by
\eqref{2.3}, it follows 
\[
\begin{aligned}
\omega\bigl(D_f\Phi_{L,t}(f)g_1,g_2\bigr)
&=
\omega\bigl(D_f\Phi_{L,t}(f)g_1,D_f\Phi_{L,t}(f)(D_f\Phi_{L,t}(f))^{-1}g_2\bigr)\\
&=
\omega\bigl(
g_1,(D_f\Phi_{L,t}(f))^{-1}g_2
\bigr).
\end{aligned}
\]
By the duality between $H_x^{-1/2}$ and
$H_x^{1/2}$ in the form
\[
\|h\|_{H^{-1/2}}
=\sup_{\|g\|_{H^{1/2}}\le1}
\left|\int_{\mathbb T^3}h\overline g\,dx\right|
=\sup_{\|g\|_{H^{1/2}}\le1}|\omega(h,g)|
,
\]
where the second equality follows by rotating each trial function $g$ if necessary, we therefore obtain
\[
\begin{aligned}
\|D_f\Phi_{L,t}(f)g_1\|_{H_x^{-1/2}}
&=
\sup_{\|g_2\|_{H_x^{1/2}}\le1}
\left|
\omega\bigl(
D_f\Phi_{L,t}(f)g_1,g_2
\bigr)
\right|\\
&\le
\|g_1\|_{H_x^{-1/2}}
\|(D_f\Phi_{L,t}(f))^{-1}\|_
{H_x^{1/2}\to H_x^{1/2}}
\lesssim
\|g_1\|_{H_x^{-1/2}},
\end{aligned}
\]
uniformly for $|t|\le h$, $f\in\mathcal U$,
and $L\in\mathbb N$, where the last inequality
follows from \eqref{2.3}. This finishes the desired proof.
\end{proof}

\subsubsection*{Acknowledgements}
Y.L. was supported by the NSF grant of Guangdong (No. 2024A1515010497), the QB-Program of Guangdong (No. 2024QN11X141) and the NSF grant of China (No. 12301301).

\subsubsection*{Data availability}
Data sharing is not applicable to this article as no datasets were generated or analysed during the current study.

\subsubsection*{Conflict of interest}
The authors declare that they have no conflict of interest.

\end{document}